\documentclass[11pt]{article}
\usepackage{amsfonts, amsmath, amssymb, latexsym, eucal, amscd,cite} 
\newtheorem{theorem}{Theorem}[section]
\newtheorem{lemma}[theorem]{Lemma}
\newtheorem{proposition}[theorem]{Proposition}
\newtheorem{corollary}[theorem]{Corollary}
\newtheorem{exAux}[theorem]{Example}

\newtheorem{Def}[theorem]{Definition}
\newenvironment{definition}{\begin{Def} \rm}{\end{Def}}
\newtheorem{Note}[theorem]{Note}

\newtheorem{Problem}[theorem]{Problem}

\newtheorem{Rem}[theorem]{Remark}

\newtheorem{Not}[theorem]{Notation}
\newenvironment{notation}{\begin{Not} \rm}{\end{Not}}
\newtheorem{Conj}[theorem]{Conjecture}

\newtheorem{Ass}[theorem]{Assumption}

\newenvironment{proof}{\medskip\noindent{\bf Proof.\ }}{\qed\medskip}
\newenvironment{proofof}[1]{\medskip\noindent{\bf Proof  of {#1}.\ 
}}{\qed\medskip}
\newcommand{\qed}{\hfill\mbox{$\Box$\qquad\qquad}}

\newcommand{\F}{\mathbb{F}}
\newcommand{\Mat}{\text{\rm Mat}}

\newcommand{\vphi}{\varphi}

\newif\ifDRAFT
\DRAFTfalse

\begin{document}

\title{An LR pair that can be extended to an LR triple}

\author{Kazumasa Nomura}

\date{}

\maketitle

\bigskip

{
\small
\begin{quote}
\begin{center}
{\bf Abstract}
\end{center}
Fix an integer $d \geq 0$, a field $\F$, and a vector space $V$ over $\F$ with 
dimension $d+1$.
By a decomposition of $V$ we mean a sequence $\{V_i\}_{i=0}^d$
of $1$-dimensional $\F$-subspaces of $V$ such that $V = \sum_{i=0}^d V_i$
(direct sum).
Consider $\F$-linear transformations $A$, $B$ from $V$ to $V$.
Then $A,B$ is called an LR pair
whenever there exists a decomposition $\{V_i\}_{i=0}^d$ of $V$
such that $A V_i = V_{i-1}$ and $B V_i = V_{i+1}$ for $0 \leq i \leq d$,
where $V_{-1}=0$ and $V_{d+1}=0$.
By an LR triple we mean a $3$-tuple $A,B,C$ of $\F$-linear transformations
from $V$ to $V$ such that any two of them form an LR pair.
In the present paper, we consider how an LR pair $A,B$ can be
extended to an LR triple $A,B,C$.
\end{quote}
}

\section{Introduction}

Throughout the paper, fix an integer $d \geq 0$, a field $\F$,
and a vector space $V$ over $\F$ with dimension $d+1$.
Let $\text{End}(V)$ denote the $\F$-algebra consisting of the
$\F$-linear transformations from $V$ to $V$,
and let $\Mat_{d+1}(\F)$ denote the $\F$-algebra consisting of the
$(d+1) \times (d+1)$ matrices that have all entries in $\F$.
We index the rows and columns by $0,1,\ldots,d$.

The notion of an LR triple was introduced by Paul Terwilliger
in \cite{T:LRT}, motivated by the study about the equitable presentation of the Lie algebra $\mathfrak{sl}_2$
and the quantum algebra $U_q(\mathfrak{sl}_2)$ (see \cite{BenT,HarT,ITW,T:uqsl2,T:Billiard}).
The LR triples were classified in \cite{T:LRT}.
See also \cite{N:LRT} about LR triples.

We recall some basic facts concerning an LR triple.
By a {\em decomposition} of $V$ we mean a sequence $\{V_i\}_{i=0}^d$ of $1$-dimensional $\F$-subspaces
of $V$ such that $V = \sum_{i=0}^d V_i$ (direct sum).
Let $A$ be an element in $\text{End}(V)$ and $\{V_i\}_{i=0}^d$ be a decomposition of $V$.
For notational convenience set $V_{-1}=0$ and $V_{d+1}=0$.
We say $A$ {\em lowers} $\{V_i\}_{i=0}^d$ whenever $AV_i = V_{i-1}$ for $0 \leq i \leq d$.
We say $A$ {\em raises} $\{V_i\}_{i=0}^d$ whenever $A V_i = V_{i+1}$ for $0 \leq i \leq d$.
We say $A$ is {\em tridiagonal on $\{V_i\}_{i=0}^d$} whenever
\begin{align*}
  A V_i &\subseteq V_{i-1} + V_i + V_{i+1}   &&  (0 \leq i \leq d).
\end{align*}
We say $A$ is {\em irreducible tridiagonal on $\{V_i\}_{i=0}^d$} whenever 
$A$ is tridiagonal on $\{V_i\}_{i=0}^d$ and 
\begin{align*}
  A V_i &\not\subseteq V_{i-1} + V_i, &
  A V_i &\not\subseteq V_i + V_{i+1} && (0 \leq i \leq d).
\end{align*}

By a {\em lowering-raising pair} (or {\em LR pair}) on $V$ we mean an ordered pair
$A,B$ of elements in  $\text{End}(V)$ such that there exists a decomposition $\{V_i\}_{i=0}^d$ of $V$
that is lowered by $A$ and raised by $B$.
In this case, the decomposition $\{V_i\}_{i=0}^d$ is uniquely determined by $A,B$
(see Lemma \ref{lem:ABdecomp}).
We call $\{V_i\}_{i=0}^d$ the {\em $(A,B)$-decomposition}.
By an {\em LR triple} on $V$ we mean an ordered $3$-tuple $A,B,C$
of elements in $\text{End}(V)$ such that any two of them form an LR pair.

In the present paper we consider how an LR pair can be extended to an LR triple.
Our results are based on the following fact:

\begin{lemma}    {\rm (See \cite[Proposition 13.39]{T:LRT}.)}
\label{lem:based}    \samepage
\ifDRAFT {\rm lem:based}. \fi
Let $A,B,C$ be an LR triple on $V$. 
Then the following hold.
\begin{itemize}
\item[\rm (i)]
Let $\{V'_i\}_{i=0}^d$ be the $(A,C)$-decomposition. Then
\begin{equation}             \label{eq:I}
\text{$A$ lowers $\{V'_i\}_{i=0}^d$ and $B$ is irreducible tridiagonal on $\{V'_i\}_{i=0}^d$}.
\end{equation}
\item[\rm (ii)]
Let $\{V''_i\}_{i=0}^d$ be the $(B,C)$-decomposition. Then
\begin{equation}             \label{eq:II}
\text{$B$ lowers $\{V''_i\}_{i=0}^d$ and $A$ is irreducible tridiagonal on $\{V''_i\}_{i=0}^d$}.
\end{equation}
\end{itemize}
\end{lemma}

In Lemma \ref{lem:based}, (ii) is obtained from 
(i) by exchanging $A$ and $B$.
For a while, we focus on (i).
Our first main result follows.

\begin{theorem}    \label{thm:main}    \samepage
\ifDRAFT {\rm thm:main}. \fi
Let $A,B$ be an LR pair on $V$ and let $\{V'_i\}_{i=0}^d$ be a decomposition of $V$.
Then the following are equivalent:
\begin{itemize}
\item[\rm (i)]
$A$ and $B$ satisfy \eqref{eq:I};
\item[\rm (ii)]
there exists an element $C$ in $\text{\rm End}(V)$ such that $A,B,C$ is an LR triple
and $\{V'_i\}_{i=0}^d$ is the $(A,C)$-decomposition.
\end{itemize}
\end{theorem}

Referring to Theorem \ref{thm:main},
later (see Proposition \ref{prop:const}) we explicitly construct an element $C$
referred to in (ii).
Below we determine all such elements $C$.
There are two kinds of LR triples:  ``bipartite'' and ``nonbipartite'' (see Definition \ref{def:bipartite}).

\begin{theorem}   \label{thm:nonbipartite}    \samepage
\ifDRAFT {\rm thm:nonbipartite}. \fi
Let $A,B,C$  and $A,B,\tilde{C}$ be LR triples on $V$.
Assume that $A,B,C$ is nonbipartite.
Then the following are equivalent:
\begin{itemize}
\item[\rm (i)]
the $(A,C)$-decomposition is equal to the $(A,\tilde{C})$-decomposition;
\item[\rm (ii)]
there exists a nonzero scalar $\gamma$ in $\F$ such that 
$\tilde{C} = \gamma C$.
\end{itemize}
\end{theorem}

\begin{theorem}    \label{thm:bipartite}    \samepage
\ifDRAFT {\rm thm:bipartite}. \fi
Let $A,B,C$  and $A,B,\tilde{C}$ be LR triples on $V$.
Assume that $A,B,C$ is bipartite.
Then the following are equivalent:
\begin{itemize}
\item[\rm (i)]
the $(A,C)$-decomposition is equal to the $(A,\tilde{C})$-decomposition;
\item[\rm (ii)]
there exist nonzero scalars 
$\gamma_\text{\rm out}$, $\gamma_\text{\rm in}$ in $\F$ such that 
  $\tilde{C} = \gamma_\text{\rm out} C_\text{\rm out} +
                \gamma_\text{\rm in} C_\text{\rm in}$,
where $C_\text{\rm out}$ and $C_\text{\rm in}$ are from Definition \ref{def:XoutXin}.
\end{itemize}
\end{theorem}

In the above, we focused on \eqref{eq:I}.
By exchanging $A$ and $B$, we obtain similar results concerning \eqref{eq:II}.
For instance, from Theorem \ref{thm:main} we obtain:

\begin{corollary}    \label{cor:main}    \samepage
\ifDRAFT {\rm cor:main}. \fi
Let $A,B$ be an LR pair on $V$ and let $\{V''_i\}_{i=0}^d$ be a decomposition of $V$.
Then the following are equivalent:
\begin{itemize}
\item[\rm (i)]
$A$ and $B$ satisfy \eqref{eq:II};
\item[\rm (ii)]
there exists an element $C$ in $\text{\rm End}(V)$ such that $A,B,C$ is an LR triple
and $\{V''_i\}_{i=0}^d$ is the $(B,C)$-decomposition.
\end{itemize}
\end{corollary}

We now consider the conditions \eqref{eq:I} and \eqref{eq:II} at the same time.
In this case, we can start with any pair of elements $A,B$ rather than an LR pair
$A,B$.
Actually, the following holds:

\begin{proposition}    \label{prop:LRpair}    \samepage
\ifDRAFT {\rm prop:LRpair}. \fi
Let $A,B$ be elements in $\text{\rm End}(V)$.
Assume that there exist  decompositions $\{V'_i\}_{i=0}^d$ and $\{V''_i\}_{i=0}^d$ of $V$
that satisfy \eqref{eq:I} and \eqref{eq:II}, respectively.
Then $A,B$ is an LR pair.
\end{proposition}

Combining Theorem \ref{thm:main}, Corollary \ref{cor:main} and Proposition \ref{prop:LRpair}
we obtain:

\begin{corollary}   \label{cor:main1}    \samepage
\ifDRAFT {\rm thm:main1}. \fi
Let $A,B$ be elements in $\text{\rm End}(V)$.
Let $\{V'_i\}_{i=0}^d$ and $\{V''_i\}_{i=0}^d$ be decompositions of $V$
that satisfy \eqref{eq:I} and \eqref{eq:II}, respectively.
Then the following hold.
\begin{itemize}
\item[\rm (i)]
There exists an element $C$ in $\text{\rm End}(V)$
such that $A,B,C$ is an LR triple and
$\{V'_i\}_{i=0}^d$ is the $(A,C)$-decomposition.
\item[\rm (ii)]
There exists an element $C'$ in $\text{\rm End}(V)$
such that $A,B,C'$ is an LR triple and
$\{V''_i\}_{i=0}^d$ is the $(B,C')$-decomposition.
\end{itemize}
\end{corollary}

In Corollary \ref{cor:main1},
the decomposition $\{V''_i\}_{i=0}^d$ may not be
equal to the $(B,C)$-decomposition.
So we add  an extra condition $V'_d = V''_d$
(see Lemma \ref{lem:coincide}):

\begin{theorem}   \label{thm:main1}    \samepage
\ifDRAFT {\rm thm:main1}. \fi
Let $A,B$ be elements in $\text{\rm End}(V)$.
Let  $\{V'_i\}_{i=0}^d$ and $\{V''_i\}_{i=0}^d$ be decompositions of $V$.
Then the following are equivalent:
\begin{itemize}
\item[\rm (i)]
the decompositions satisfy $V'_d = V''_d$ and  \eqref{eq:I}, \eqref{eq:II};
\item[\rm (ii)]
there exists an element $C$ in $\text{\rm End}(V)$
such that 
\begin{itemize}
\item[\rm (a)]
$A,B,C$ is an LR triple;
\item[\rm (b)]
$\{V'_i\}_{i=0}^d$ is the $(A,C)$-decomposition;
\item[\rm (c)]
$\{V''_i\}_{i=0}^d$ is the $(B,C)$-decomposition.
\end{itemize}
\end{itemize}
\end{theorem}

Let $\tilde{C}$ be an element in $\text{End}(V)$.
With reference to Theorem \ref{thm:main1},
consider the conditions (a)--(c), in which
$C$ is replaced by $\tilde{C}$:
\begin{itemize}     \samepage
\item[$(\tilde{\rm a})$]
$A,B, \tilde{C}$ is an LR triple;
\item[$(\tilde{\rm b})$]
$\{V'_i\}_{i=0}^d$ is the $(A, \tilde{C})$-decomposition;
\item[$(\tilde{\rm c})$]
$\{V''_i\}_{i=0}^d$ is the $(B, \tilde{C})$-decomposition.
\end{itemize}

\begin{theorem}    \label{thm:main2}        \samepage
\ifDRAFT {\rm thm:main2}. \fi
With reference to Theorem \ref{thm:main1},
assume that the equivalent conditions {\rm (i)} and {\rm (ii)} hold.
Fix an element $C$ in $\text{\rm End}(V)$ that satisfies {\rm (a)--(c)}.
Let $\tilde{C}$ be an element in $\text{\rm End}(V)$.
Then the following hold.
\begin{itemize}
\item[\rm (i)]
Assume that $A,B,C$ is nonbipartite.
Then $\tilde{C}$ satisfies $(\tilde{\rm a})$--$(\tilde{\rm c})$ if and only if
there exists a nonzero scalar $\gamma$ in $\F$ such that
$\tilde{C} = \gamma C$.
\item[\rm (ii)]
Assume that $A,B,C$ is bipartite.
Then $\tilde{C}$ satisfies  $(\tilde{\rm a})$--$(\tilde{\rm c})$ if and only if
there exist nonzero scalars $\gamma_\text{\rm out}$, $\gamma_\text{\rm in}$
in $\F$ such that
$\tilde{C} =
  \gamma_\text{\rm out} C_\text{\rm out} + \gamma_\text{\rm in} C_\text{\rm in}$.
\end{itemize}
\end{theorem}

The paper is organized as follows.
In Section \ref{sec:LRpair} we recall some basic facts concerning an LR pair.
In Section \ref{sec:prop} we prove Proposition \ref{prop:LRpair}.
In Section \ref{sec:Toeplitz} we recall the Toeplitz matrices,
and prove some lemmas for later use.
In Section \ref{sec:thmmain} we prove Theorem \ref{thm:main}.
In Section \ref{sec:thmmain1} we prove Theorem \ref{thm:main1}.
In Section \ref{sec:bipartite} we recall some facts concerning  a bipartite LR triple.
In Section \ref{sec:pre} we recall some data for an LR triple.
In Section \ref{sec:proofs} we prove Theorems \ref{thm:nonbipartite},
\ref{thm:bipartite}, \ref{thm:main2}.

\section{LR pairs}
\label{sec:LRpair}

In this section we recall some basic facts concerning an LR pair.
For $X$ in $\text{End}(V)$, we denote the kernel of $X$ by $\text{Ker}\, X$.
For a decomposition $\{V_i\}_{i=0}^d$ of $V$, we define $V_{-1}=0$ and $V_{d+1}=0$.

\begin{lemma}    \label{lem:Alowers}    \samepage
\ifDRAFT {\rm lem:Alowers}. \fi
Let $A$ be an element in $\text{\rm End}(V)$,
and let $\{V_i\}_{i=0}^d$ be a decomposition of $V$
that is lowered by $A$.
Then the following hold.
\begin{itemize}
\item[\rm (i)]
For $0 \leq i \leq j \leq d$,  $A^i V_j = V_{j-i}$.
Moreover, the restriction of $A^i$ on $V_j$ induces a bijection from $V_j$ to $V_{j-i}$.
\item[\rm (ii)]
$V_0 = \text{\rm Ker} \, A$.
\end{itemize}
\end{lemma}

\begin{proof}
(i):
Clearly we have $A^i V_j = V_{j-i}$.
The second assertion follows since each of $V_j$, $V_{j-i}$ has dimension one.
 
(ii):
We have $V_0 \subseteq \text{Ker} \, A$ since $A V_0 = 0$.
To see the inverse inclusion, pick any $v \in \text{Ker}\, A$.
Write $v = \sum_{i=0}^d v_i$ with $v_i \in V_i$ $(0 \leq i \leq d)$.
Then $0 = A v = \sum_{i=1}^d A v_i$ and $A v_i  \in V_{i-1}$ for $1 \leq i \leq d$.
So $A v_i = 0$ for $1 \leq i \leq d$.
This forces $v_i = 0$ for $1 \leq i \leq d$, since $A$ induces the bijection from $V_i$ to $V_{i-1}$.
Now $v = v_0$, so $v \in V_0$.
Thus $V_0 \supseteq \text{Ker} \, A$.
The result follows.
\end{proof}

\begin{lemma}    \label{lem:Vicoincide}     \samepage
\ifDRAFT {\rm lem:Vicoincide}. \fi
Let $A$ be an element in $\text{\rm End}(V)$.
Let $\{V_i\}_{i=0}^d$ and $\{\tilde{V}_i\}_{i=0}^d$ be decompositions of $V$,
each of which is lowered by $A$.
Assume $V_d = \tilde{V}_d$.
Then $V_i = \tilde{V}_i$ for $0 \leq i \leq d$.
\end{lemma}

\begin{proof}
By Lemma \ref{lem:Alowers}(i) $V_i = A^{d-i} V_d$ for $0 \leq i \leq d$.
Similarly $\tilde{V}_i = A^{d-i} \tilde{V}_d$ for $0 \leq i \leq d$.
The result follows from these comments and the assumption $V_d = \tilde{V}_d$.
\end{proof}

\begin{lemma}    \label{lem:Braises}    \samepage
\ifDRAFT {\rm lem:Braises}. \fi
Let $B$ be an element in $\text{\rm End}(V)$,
and let $\{V_i\}_{i=0}^d$ be a decomposition of $V$
that is raised by $B$.
Then the following hold.
\begin{itemize}
\item[\rm (i)]
For $0 \leq i,j \leq d$ with $i+j \leq d$,　
$B^i V_j = V_{i+j}$.
Moreover, the restriction of $B^i$ on $V_j$ induces a bijection from $V_j$ to $V_{i+j}$.
\item[\rm (ii)]
$V_d = \text{\rm Ker} \, B$.
\end{itemize}
\end{lemma}

\begin{proof}
Similar to the proof of Lemma \ref{lem:Alowers}.
\end{proof}

\begin{lemma}    \label{lem:ABdecomp}    \samepage
\ifDRAFT {\rm lem:ABdecomp}. \fi
Let $A,B$ be an LR pair on $V$.
Then there exists a unique decomposition of $V$
that is lowered by $A$ and raised by $B$.
\end{lemma}

\begin{proof}
Let $\{V_i\}_{i=0}^d$ and $\{\tilde{V}_i\}_{i=0}^d$ be a decompositions of $V$, each of which
is lowered by $A$ and raised by $B$.
By Lemma \ref{lem:Braises}(ii), $V_d = \text{Ker} \, B = \tilde{V}_d$.
By this and Lemma \ref{lem:Vicoincide} $V_i = \tilde{V}_i$ for $0 \leq i \leq d$.
\end{proof}

\begin{definition}   \label{def:ABdecomp}    \samepage
\ifDRAFT {\rm def:ABdecomp}. \fi
Let $A,B$ be an LR pair on $V$.
By the {\em $(A,B)$-decomposition} we mean the decomposition of $V$
that is lowered by $A$ and raised by $B$.
\end{definition}

\begin{lemma}    {\rm (See \cite[Lemma 3.6]{T:LRT}.)}
\label{lem:BAdecomp}    \samepage
\ifDRAFT {\rm lem:BAdecomp}. \fi
Let $A,B$ be an LR pair on $V$ with $(A,B)$-decomposition $\{V_i\}_{i=0}^d$.
Then $B,A$ is an LR pair with $(B,A)$-decomposition $\{V_{d-i}\}_{i=0}^d$.
\end{lemma}

\begin{lemma}    {\rm (See \cite[Lemma 3.12]{T:LRT}.)}
\label{lem:parameterseq}   \samepage
\ifDRAFT {\rm lem:parameterseq}. \fi
Let $A,B$ be an LR pair on $V$ with $(A,B)$-decomposition $\{V_i\}_{i=0}^d$.
Then, for $0 \leq i \leq d$, the $\F$-subspace $V_i$ is invariant under $AB$ and $BA$.
Moreover, for $1 \leq i \leq d$,
the eigenvalue of $AB$ on $V_{i-1}$ is nonzero and equal to the eigenvalue of $BA$
on $V_i$.
\end{lemma}

\begin{definition}     {\rm (See \cite[Definition 3.13]{T:LRT}.)}
\label{def:parameterseq}    \samepage
\ifDRAFT {\rm def:parameterseq}. \fi
Let $A,B$ be an LR pair on $V$.
By the {\em parameter sequence} for $A,B$ we mean the sequence $\{\vphi_i\}_{i=1}^d$,
where $\vphi_i$ denotes the eigenvalue of $BA$ on $V_i$.
We emphasize that $\vphi_i \neq 0$ for $1 \leq i \leq d$.
\end{definition}

\begin{lemma}    {\rm (See \cite[Lemma 3.14]{T:LRT}.)}
\label{lem:BAparam}    \samepage
\ifDRAFT {\rm lem:BAparam}. \fi
Let $A,B$ be an LR pair with parameter sequence $\{\vphi_i\}_{i=1}^d$.
Then the LR pair $B,A$ has  parameter sequence $\{\vphi_{d-i+1}\}_{i=1}^d$.
\end{lemma}

\begin{definition}    {\rm (See \cite[Definition 3.21]{T:LRT}.)}
\label{def:ABbasis}    \samepage
\ifDRAFT {\rm def:ABbasis}. \fi
Let $A,B$ be an LR pair on $V$ with $(A,B)$-decomposition $\{V_i\}_{i=0}^d$.
Observe that there exists a basis $\{v_i\}_{i=0}^d$ for $V$ such that
$v_i \in V_i$ and $A v_i = v_{i-1}$ for $0 \leq i \leq d$, 
where $v_{-1}=0$.
We call $\{v_i\}_{i=0}^d$ an {\em $(A,B)$-basis} for $V$.
\end{definition}

\begin{lemma}   {\rm (See \cite[Lemma 3.24]{T:LRT}.)}
  \label{lem:ABbasis}   \samepage
\ifDRAFT {\rm lem:ABbasis}. \fi
Let $A,B$ be an LR pair on $V$ with parameter sequence $\{\vphi_i\}_{i=1}^d$.
Let $\{v_i\}_{i=0}^d$ be an $(A,B)$-basis for $V$.
Then 
\begin{align*}
  B v_i &= \vphi_{i+1} v_{i+1}   &&  (0 \leq i \leq d),
\end{align*}
where $v_{d+1}=0$ and $\vphi_{d+1}$ is an indeterminate.
\end{lemma}

\begin{lemma}    {\rm (See \cite[Lemmas 3.28, 3.30]{T:LRT}.)}
\label{lem:BAbasis}   \samepage
\ifDRAFT {\rm lem:BAbais}. \fi
Let $A,B$ be an LR pair on $V$ with parameter sequence $\{\vphi_i\}_{i=1}^d$.
Let $\{v'_i\}_{i=0}^d$ be an $(B,A)$-basis for $V$.
Then
\begin{align*}
  A v'_i &= \vphi_{d-i} v'_{i+1}    &&  (0 \leq i \leq d),
\\
  B v'_i &= v'_{i-1}                  &&  (0 \leq i \leq d),
\end{align*}
where $v'_{-1}=0$, $v'_{d+1}=0$, and $\vphi_0$ is an indeterminate.
\end{lemma}

\begin{lemma}    \label{lem:AB}    \samepage
\ifDRAFT {\rm lem:AB}. \fi
Let $A,B$ and $A,\tilde{B}$ be LR pairs on $V$.
Assume that the following hold.
\begin{itemize}
\item[\rm (i)]
The $(A, B)$-decomposition is equal to the $(A, \tilde{B})$-decomposition.
\item[\rm (ii)]
The LR pairs $A,B$ and $A, \tilde{B}$ have the same parameter sequence.
\end{itemize}
Then $B =\tilde{B}$.
\end{lemma}

\begin{proof}
Let $\{v_i\}_{i=0}^d$ be an $(A,B)$-basis for $V$.
By (i) and Definition \ref{def:ABbasis}, $\{v_i\}_{i=0}^d$ is an $(A,\tilde{B})$-basis for $V$.
Now by (ii) and Lemma \ref{lem:ABbasis}, on the basis $\{v_i\}_{i=0}^d$
the action of $B$ coincides with the action of $\tilde{B}$.
Thus $B = \tilde{B}$.
\end{proof}

\begin{definition}    \label{def:idempotentseq}    \samepage
\ifDRAFT {\rm def:idempotentseq}. \fi
Let $\{V_i\}_{i=0}^d$ be a decomposition of $V$.
For $0 \leq i \leq d$ define $E_i \in \text{End}(V)$ such that
$(E_i - I)V_i = 0$ and $E_i V_j=0$ for $0 \leq j \leq d$, $j \neq i$,
where $I$ denotes the identity.
We call $\{E_i\}_{i=0}^d$ the {\em idempotent sequence} for $\{V_i\}_{i=0}^d$.
\end{definition}

\begin{definition}  {\rm (See \cite[Definition 3.5]{T:LRT}.)}
  \label{def:LRpairidemp}    \samepage
\ifDRAFT {\rm def:LRpairidemp}. \fi
Let $A,B$ be an LR pair on $V$.
By the {\em idempotent sequence} for $A,B$
we mean the idempotent sequence for the $(A,B)$-decomposition.
\end{definition}

\section{Proof of Proposition \ref{prop:LRpair}}
\label{sec:prop}

In this section we prove Proposition \ref{prop:LRpair}.
We begin by recalling the notion of a flag.
By a {\em flag} on $V$ we mean a sequence $\{U_i\}_{i=0}^d$ of $\F$-subspaces of $V$
such that
$\dim U_i = i+1$ $(0 \leq i \leq d)$ and $U_{i-1} \subseteq U_i$ $(1 \leq i \leq d)$.
Let $A$ be an element in $\text{End}(V)$, and 
let $\{U_i\}_{i=0}^d$ be a flag on $V$.
For notational convenience, set $U_{-1}=0$.
We say {\em $A$ lowers $\{U_i\}_{i=0}^d$} whenever 
\begin{align*}
  A U_i &= U_{i-1}   && (0 \leq i \leq d).  
\end{align*}
We say {\em $A$ raises $\{U_i\}_{i=0}^d$} whenever
\begin{align*}
  A U_i &\subseteq U_{i+1},  &
  A U_i &\not\subseteq U_i  &&  (0 \leq i < d). 
\end{align*}
Let $\{V_i\}_{i=0}^d$ be a decomposition of $V$.
We say {\em $\{U_i\}_{i=0}^d$ is induced by $\{V_i\}_{i=0}^d$}
whenever 
\begin{align*}
   U_i &= V_0 + V_1 + \cdots + V_i  &&  (0 \leq i \leq d).  
\end{align*}

\begin{lemma}   \label{lem:flaglower}    \samepage
\ifDRAFT {\rm lem:flaglower}. \fi
Let $\{V_i\}_{i=0}^d$ be a decomposition of $V$,
and let $\{U_i\}_{i=0}^d$ be the flag that is induced by $\{V_i\}_{i=0}^d$.
Let $A$ be an element in $\text{\rm End}(V)$.
Then the following hold.
\begin{itemize}
\item[\rm (i)]
Assume that $A$ lowers $\{V_i\}_{i=0}^d$.  Then $A$ lowers $\{U_i\}_{i=0}^d$.
\item[\rm (ii)]
Assume that $A$ is irreducible tridiagonal on $\{V_i\}_{i=0}^d$. 
Then  $A$ raises $\{U_i\}_{i=0}^d$.
\end{itemize}
\end{lemma}

\begin{proof}
(i): Clear by the construction.

(ii):
We have $A V_i \subseteq V_{i-1} + V_i + V_{i+1}$ for $0 \leq i <d$
since $A$ is tridiagonal on $V$.
So $A U_i \subseteq U_{i+1}$ for $0 \leq i < d$,
We have 
$A V_i \not\subseteq V_{i-1} + V_i$ for $0 \leq i < d$,
since $A$ is irreducible tridiagonal on $\{V_i\}_{i=0}^d$.
By these comments we find that $A U_i \not\subseteq U_i$ for $0 \leq i < d$.
The result follows.
\end{proof}

\begin{lemma}    \label{lem:irred}        \samepage
\ifDRAFT {\rm lem:irred}. \fi
Let $A,B$ be elements in $\text{\rm End}(V)$.
Assume that there exists a flag $\{U_i\}_{i=0}^d$ on $V$ that is lowered by $A$
and raised by $B$.
Then there is no $\F$-subspace $W \subseteq V$ such that
$W \neq 0$, $W \neq V$, $AW \subseteq W$, $BW \subseteq W$.
\end{lemma}

\begin{proof}
Let $W$ be an $\F$-subspace of $V$ such that
$W \neq 0$, $AW \subseteq W$, $B W \subseteq W$.
It suffices to show $W=V$.
To this end, we show
\begin{align}
    U_i &\subseteq W            \label{eq:UiW}
\end{align}
using induction on $i=0,1,\ldots,d$.
We first show \eqref{eq:UiW} for $i=0$.
Pick any nonzero $w \in W$.
There exists an integer $s$ $(0 \leq s \leq d)$ such that
$w \not\in U_{s-1}$ and $w \in U_s$.
Then $A^s w \in W$ since $AW \subseteq W$.
Moreover $0 \neq A^s w \in A^s U_s = U_0$
since $A$ lowers $\{U_i\}_{i=0}^d$.
So $A^s w$ is a nonzero element in $W \cap U_0$.
This forces $U_0 = \F A^s w$ since $U_0$ has dimension $1$,
and so \eqref{eq:UiW} holds for $i=0$.
Now assume $1 \leq i \leq d$. 
By induction $U_{i-1} \subseteq W$.
In this inclusion, apply $B$ to each side and use $B W \subseteq W$ to find
$B U_{i-1} \subseteq W$.
We have $B U_{i-1} \subseteq U_i$ and  $B U_{i-1} \not\subseteq U_{i-1}$ since 
$B$ raises $\{U_i\}_{i=0}^d$.
Pick an element $u \in B U_{i-1}$ such that $u \not\in U_{i-1}$.
Then $u \in W$, $u \in U_i$ and $u \not\in U_{i-1}$.
This forces $U_i = U_{i-1} + \F u$, since  $\dim U_{i} = \dim U_{i-1} + 1$.
 Moreover $U_{i-1} + \F u \subseteq W$.
Thus \eqref{eq:UiW} holds.
The result follows.
\end{proof}

Let $\{U_i\}_{i=0}^d$ and $\{U'_i\}_{i=0}^d$ be flags on $V$.
We say these two flags are {\em opposite} whenever
\begin{align*}
  U_i \cap U'_j &=0  \qquad  \text{ if } \; i+j<d    && (0 \leq i,j \leq d).
\end{align*}

\begin{lemma}   \label{lem:opposite2}    \samepage
\ifDRAFT {\rm lem:opposite2}. \fi
Let $A,B$ be elements in $\text{\rm End}(V)$,
and let $\{U'_i\}_{i=0}^d$ and $\{U''_i\}_{i=0}^d$ be flags on $V$
such that
\begin{itemize}
\item[\rm (i)]
$A$ lowers $\{U'_i\}_{i=0}^d$ and raises $\{U''_i\}_{i=0}^d$;
\item[\rm (ii)]
$B$ raises $\{U'_i\}_{i=0}^d$ and lowers $\{U''_i\}_{i=0}^d$.
\end{itemize}
Then the flags $\{U'_i\}_{i=0}^d$ and $\{U''_i\}_{i=0}^d$ are opposite.
\end{lemma}

\begin{proof}
For $0 \leq i \leq d-1$ define
\begin{align*}
        W_i &= U'_i \cap U''_{d-i-1}.
\end{align*}
It suffices to show that $W_i = 0$ for $0 \leq i \leq d-1$.
For notational convenience, set $W_{-1}=0$ and $W_d = 0$.
By (i) 
\begin{align}
   A W_i & \subseteq W_{i-1}   &&  (0 \leq i \leq d-1).                \label{eq:AWi}
\end{align}
By (ii)
\begin{align}
   B W_i & \subseteq W_{i+1}   &&  (0 \leq i \leq d-1).                \label{eq:BWi}
\end{align}
Now define
\[
  W = W_0 + W_1 + \cdots + W_{d-1}.
\]
By \eqref{eq:AWi} and \eqref{eq:BWi} we find that 
$A W \subseteq W$ and $B W \subseteq W$.
By these comments and Lemma \ref{lem:irred},
either $W=0$ or $W=V$.
We have $W \neq V$ since $W \subseteq U'_{d-1}$.
So $W=0$, and this forces $W_i=0$ for $0 \leq i \leq d-1$.
\end{proof}

\begin{lemma}    \label{lem:opposite}       \samepage
\ifDRAFT {\rm lem:opposite}. \fi
Let $\{U_i\}_{i=0}^d$ and $\{U'_i\}_{i=0}^d$ be opposite flags on $V$.
For $0 \leq i \leq d$ define 
\begin{align}
  Z_i = U_i \cap U'_{d-i}.                \label{eq:defZi}
\end{align}
Then $\{Z_i\}_{i=0}^d$ is a decomposition of $V$.
\end{lemma}

\begin{proof}
We have $Z_i \neq 0$ for $0 \leq i \leq d$, 
since 
\[
   \dim U_i + \dim U'_{d-i}  > \dim V.
\]
We show that
\begin{equation}
     Z_0 + Z_1 + \cdots + Z_d              \label{eq:sum}
\end{equation}
is a direct sum.
To this end, we show that
\begin{align}
  (Z_0 + Z_1 + \cdots + Z_{i-1}) \cap Z_i = 0    \label{eq:cap}
\end{align}
for $1 \leq i \leq d$. Pick any $i$ $(1 \leq i \leq d)$.
By \eqref{eq:defZi}
\begin{align*}
  Z_0 + Z_1 + \cdots + Z_{i-1} &\subseteq U_{i-1},    &
  Z_i &\subseteq U'_{d-i}.
\end{align*}
So the left-hand side of \eqref{eq:cap} is contained in $U_{i-1} \cap U'_{d-i}$.
We have $U_{i-1} \cap U'_{d-i}=0$ since the two flags are opposite.
By these comments \eqref{eq:cap} holds, and so \eqref{eq:sum} is a direct sum.
The result follows from this and that $Z_i \neq 0$ for $0 \leq i \leq d$. 
\end{proof}

\begin{lemma}    \label{lem:opposite3}    \samepage
\ifDRAFT {\rm lem:opposite3}. \fi
With reference to Lemma \ref{lem:opposite2}, consider the decomposition
$\{Z_i\}_{i=0}^d$ of $V$ referred to in Lemma \ref{lem:opposite}.
Then $\{Z_i\}_{i=0}^d$ is lowered by $A$ and raised by $B$.
\end{lemma}

\begin{proof}
For notational convenience, set $Z_{-1}=0$ and $Z_{d+1}=0$.
We have
\begin{align}            \label{eq:AZi}
          A Z_i  &\subseteq  Z_{i-1}      &&  (0 \leq i \leq d),
\end{align}
since $A$ lowers $\{U'_i\}_{i=0}^d$ and raises $\{U''_i\}_{i=0}^d$.
We have
\begin{align}            \label{eq:BZi}
          B Z_i  &\subseteq  Z_{i+1}      &&  (0 \leq i \leq d),
\end{align}
since $B$ raises $\{U'_i\}_{i=0}^d$ and lowers $\{U''_i\}_{i=0}^d$.
We claim that $A Z_i \neq 0$ for $1 \leq i \leq d$.
By way of contradiction, assume  $A Z_r = 0$ for some $r$ $(1 \leq r \leq d)$.
Consider the $\F$-subspace
\[
    W = Z_r + Z_{r+1} + \cdots + Z_d.
\]
Note that $W \neq 0$.
We have $A W \subseteq W$ by \eqref{eq:AZi} and since $A Z_r=0$.
We have $B W \subseteq W$ by \eqref{eq:BZi}.
By these comments and  Lemma \ref{lem:irred}, we must have
$W = V$, a contradiction.
We have shown the claim.
By \eqref{eq:AZi} and the claim, $A Z_i = Z_{i-1}$ for $0 \leq i \leq d$,
since $\dim Z_{i-1} = 1$.
In a similar way, we can show that $B Z_i = Z_{i+1}$ for $0 \leq i \leq d$.
Therefore the decomposition $\{Z_i\}_{i=0}^d$ is lowered by $A$ and raised by $B$.
\end{proof}

\begin{proofof}{Proposition \ref{prop:LRpair}}
Let $\{U'_i\}_{i=0}^d$ (resp.\ $\{U''_i\}_{i=0}^d$)
be the flag that is induced by $\{V'_i\}_{i=0}^d$ (resp.\ $\{V''_i\}_{i=0}^d$).
By \eqref{eq:I} and Lemma \ref{lem:flaglower}, 
$A$ lowers $\{U'_i\}_{i=0}^d$ and $B$ raises $\{U'_i\}_{i=0}^d$.
By \eqref{eq:II} and Lemma \ref{lem:flaglower},
$B$ lowers $\{U''_i\}_{i=0}^d$  and $A$ raises $\{U''_i\}_{i=0}^d$.
By these comments and Lemma \ref{lem:opposite2},
the flags $\{U'_i\}_{i=0}^d$ and $\{U''_i\}_{i=0}^d$ are opposite.
By this and Lemmas \ref{lem:opposite}, \ref{lem:opposite3},
there exists a decomposition $\{Z_i\}_{i=0}^d$ that is lowered by $A$
and raised by $B$.
Therefore $A,B$ is an LR pair.
\end{proofof}

\section{Toeplitz matrices}
\label{sec:Toeplitz}

Let $T$ be an upper triangular matrix in $\Mat_{d+1}(\F)$.
Then $T$ is said to be {\em Toeplitz, with parameters $\{\alpha_i\}_{i=0}^d$}
whenever $T$ has $(i,j)$-entry $\alpha_{j-i}$ for $0 \leq i \leq j \leq d$.
In this case
\[
T =
 \begin{pmatrix}
  \alpha_0 & \alpha_1 & \cdot & \cdot & \cdot & \alpha_d \\
              & \alpha_0 & \alpha_1 & \cdot & \cdot & \cdot  \\
              &             & \alpha_0 & \cdot & \cdot & \cdot \\
              &             &             & \cdot & \cdot & \cdot \\
              &             &             &         & \cdot & \alpha_1 \\
    \text{\bf 0}   &     &             &         &  & \alpha_0
 \end{pmatrix}.
\]
Note that $T$ is invertible if and only if $\alpha_0 \neq 0$.
In this case,  $T^{-1}$ is upper triangular and Toeplitz (see \cite[Section 12]{T:LRT}).

\begin{lemma} {\rm (See \cite[Proposition 12.8]{T:LRT}.) }  \label{lem:Toeplitz}        \samepage
\ifDRAFT {\rm lem:Toeplitz}. \fi
Let $\{u_i\}_{i=0}^d$ and $\{v_i\}_{i=0}^d$ be bases for $V$.
Then the following are equivalent:
\begin{itemize}
\item[\rm (i)]
there exists an element $A$ in $\text{\rm End}(V)$ such that
\begin{align*}
 A u_i &= u_{i-1}  \qquad (0 \leq i \leq d), &
 A v_i &= v_{i-1}  \qquad (0 \leq i \leq d),
\end{align*}
where $u_{-1}=0$ and $v_{-1}=0$;
\item[\rm (ii)]
the transition matrix from $\{u_i\}_{i=0}^d$ to $\{v_i\}_{i=0}^d$
is upper triangular and Toeplitz.
\end{itemize}
\end{lemma}

Let $M$ be a matrix in $\Mat_{d+1}(\F)$.
We denote by  $M^{\sf T}$ the transpose of $M$.
By the {\em anti-diagonal transpose} of $M$ we mean the matrix in $\Mat_{d+1}(\F)$
that has $(i,j)$-entry $M_{d-j,d-i}$ for $0 \leq i,j \leq d$.
Let $Z$ denote the matrix in $\Mat_{d+1}(\F)$ with $(i,j)$-entry $\delta_{i+j,d}$:
\[
 Z =
 \begin{pmatrix}
   \text{\bf 0} & & & & 1 \\
     & & & 1 \\
    & & \cdot \\
    & \cdot \\
   1 & & & & \text{\bf 0}
 \end{pmatrix}
\]
Note that $Z^2=I$, where $I$ denotes the identity matrix in $\Mat_{d+1}(\F)$.

\begin{lemma}    \label{lem:anti0}    \samepage
\ifDRAFT {\rm lem:anti0}. \fi
Let $M$ be a matrix in $\Mat_{d+1}(\F)$.
\begin{itemize}
\item[\rm (i)]
The anti-diagonal transpose of $M$ is equal to $Z M^{\sf T} Z$.
\item[\rm (ii)]
Assume $T$ is upper triangular and Toeplitz.
Then the anti-diagonal transpose of $T$ is equal to $T$.
\end{itemize}
\end{lemma}

\begin{proof}
Routine verification.
\end{proof}

\begin{lemma}   \label{lem:anti}    \samepage
\ifDRAFT {\rm lem:anti}. \fi
Let $T$ and $T'$ be upper triangular Toeplitz matrices in $\Mat_{d+1}(\F)$.
Let $M$ be a matrix in $\Mat_{d+1}(\F)$, and let $M'$ be the anti-diagonal
transpose of $M$.
Then  the anti-diagonal transpose of $T M T'$ is equal to  $T' M' T$.
\end{lemma}

\begin{proof}
By Lemma \ref{lem:anti0} 
\begin{align*}
 M' &= Z M^{\sf T} Z,  &  T &= Z {T}^{\sf T} Z, &  T' &= Z {T'}^{\sf T} Z.
\end{align*}
Using these relations and $Z^2 =I$,
we compute the anti-diagonal transpose of $T M T'$ as follows.
\begin{align*}
 Z (T M T')^{\sf T} Z 
 = Z {T'}^{\sf T} M^{\sf T} T^{\sf T} Z  
 = Z {T'}^{\sf T} Z Z M^{\sf T} Z Z {T}^{\sf T} Z 
 = T' M' T.
\end{align*}
The result follows.
\end{proof}

For scalars $\{\vphi_i\}_{i=1}^d$ in $\F$,
let $S_{\vphi_1,\ldots,\vphi_d}$ denote the
the matrix in $\Mat_{d+1}(\F)$ that has subdiagonal entries $\vphi_1,\ldots,\vphi_d$
and all other entries $0$:
\begin{equation*}          
S_{\vphi_1,\ldots,\vphi_d} = 
 \begin{pmatrix}
  0 &  & & & & \text{\bf 0}  \\
  \vphi_1  & 0 &  \\
     & \vphi_2  & \cdot \\
     &   &  \cdot & \cdot \\
    &   &   &   \cdot & \cdot \\
  \text{\bf 0} &  & & & \vphi_d & 0
  \end{pmatrix}.
\end{equation*}

\begin{lemma}   \label{lem:Toeplitz2}        \samepage
\ifDRAFT {\rm lem:Toeplitz2}. \fi
Let $T$ be an invertible upper triangular Toeplitz matrix 
in $\Mat_{d+1}(\F)$.
Let $\{\vphi_i\}_{i=1}^d$ be scalars in $\F$.
Then
the anti-diagonal transpose of $T S_{\vphi_1,\ldots,\vphi_d}  T^{-1}$
is equal to  $T^{-1} S_{\vphi_d,\ldots,\vphi_1}  T$.
\end{lemma}

\begin{proof}
Observe that the anti-diagonal transpose of $S_{\vphi_1,\ldots, \vphi_d}$ is
$S_{\vphi_d, \ldots, \vphi_1}$.
Recall that $T^{-1}$ is upper triangular and Toeplitz.
Now the result follows from Lemma \ref{lem:anti}.
\end{proof}

\begin{lemma}  \label{lem:C}        \samepage
\ifDRAFT {\rm lem:C}. \fi
Let $A,B,C$ be elements in $\text{\rm End}(V)$.
Let $\{v_i\}_{i=0}^d$ and $\{v'_i\}_{i=0}^d$ be bases for $V$.
For notational convenience, set $v_{-1}=0$, $v_{d+1}=0$, $v'_{-1}=0$, $v'_{d+1}=0$.
Let $\{\vphi_i\}_{i=1}^d$ be scalars in $\F$.
Assume 
\begin{align*}
   A v_i &= v_{i-1}   && (0 \leq i \leq d),
\\
   A v'_i &= v'_{i-1}   && (0 \leq i \leq d),
\\
   B v_i & = \vphi_{i+1} v_{i+1} &&  (0 \leq i \leq d), 
\\
   C v'_i & = \vphi_{d-i} v'_{i+1} && (0 \leq i \leq d).
\end{align*}
Then the matrix representing $C$ with respect to $\{v_i\}_{i=0}^d$ is the
anti-diagonal transpose of the matrix representing $B$ with respect to $\{v'_i\}_{i=0}^d$.
\end{lemma}

\begin{proof}
Let $T$ be the transition matrix from $\{v'_i\}_{i=0}^d$ to $\{v_i\}_{i=0}^d$.
By Lemma \ref{lem:Toeplitz} $T$ is upper triangular and Toeplitz.
For an element $X$ in $\text{End}(V)$, let  $X^\natural$ (resp.\ $X^\flat$) denote
the matrix that represents $X$ with respect to $\{v_i\}_{i=0}^d$
(resp.\ $\{v'_i\}_{i=0}^d$).
By the construction, 
$X^\flat T = T X^\natural$.
Observe that
\begin{align*}
 B^\natural  &= S_{\vphi_1,\ldots,\vphi_d},  &
 C^\flat  &=  S_{\vphi_d,\ldots,\vphi_1}.
\end{align*}
Thus we find 
\begin{align*}
  B^\flat &= T S_{\vphi_1,\ldots,\vphi_d} T^{-1},  &
  C^\natural &= T^{-1} S_{\vphi_d,\ldots,\vphi_1} T. 
\end{align*}
By these comments and Lemma \ref{lem:Toeplitz2},
$C^\natural$ is the anti-diagonal transpose of $B^\flat$.
The result follows.
\end{proof}

\section{Proof of Theorem \ref{thm:main}}
\label{sec:thmmain}

In this section we prove Theorem \ref{thm:main}.
We use the following terms.
A square matrix is said to be {\em tridiagonal} whenever each nonzero entry lies either
on the diagonal, the subdiagonal, or the superdiagonal.
A tridiagonal matrix is said to be {\em irreducible} whenever each subdiagonal entry is nonero
and each superdiagonal entry is nonzero.

\begin{lemma}    \label{lem:trid}   \samepage
\ifDRAFT {\rm lem:trid}. \fi
Let $\{V_i\}_{i=0}^d$ be a decomposition of $V$,
and let $\{v_i\}_{i=0}^d$ be a basis for $V$ such that
$v_i \in V_i$ for $0 \leq i \leq d$.
Let $A$ be an element of $\text{\rm End}(V)$.
Then the following are equivalent:
\begin{itemize}
\item[\rm (i)]
$A$ is irreducible tridiagonal on $\{V_i\}_{i=0}^d$;
\item[\rm (ii)]
with respect to $\{v_i\}_{i=0}^d$,
the matrix representing $A$  is irreducible tridiagonal.
\end{itemize}
\end{lemma}

\begin{proof}
Routine verification.
\end{proof}

\begin{proposition}   \label{prop:const}    \samepage
\ifDRAFT {\rm prop:const}. \fi
Let $A,B$ be an LR  pair on $V$ with parameter sequence $\{\vphi_i\}_{i=1}^d$.
Let $\{V'_i\}_{i=0}^d$ be a decomposition of $V$ that satisfies \eqref{eq:I},
and let  $\{v'_i\}_{i=0}^d$ be a basis for $V$ such that $v'_i \in V'_i$ $(0 \leq i \leq d)$ and
\begin{align}
  A v'_i &= v'_{i-1}  &&  (0 \leq i \leq d),       \label{eq:Avi2}
\end{align}
where $v'_{-1}=0$.
Define $C \in \text{\rm End}(V)$ such that
\begin{align}
  C v'_i &= \vphi_{d-i} v'_{i+1}  &&  (0 \leq i \leq d),    \label{eq:Cvi2}
\end{align}
where $v'_{d+1}=0$ and $\vphi_0$ is an indeterminate.
Then $A,B,C$ is an LR triple and $\{V'_i\}_{i=0}^d$ is the
$(A,C)$-decomposition.
Moreover, 
the parameter sequence for the LR pair $C,A$ is equal to the
parameter sequence for $A,B$.
\end{proposition}

\begin{proof}
Let $\{V_i\}_{i=0}^d$ be the $(A,B)$-decomposition, and let
$\{v_i\}_{i=0}^d$ be the $(A,B)$-basis for $V$ such that $v_0 = v'_0$.
By Definition \ref{def:ABbasis},
\begin{align*}
  A v_i &= v_{i-1}  &&  (0 \leq i \leq d), 
\end{align*}
where $v_{-1} = 0$.
By Lemma \ref{lem:ABbasis}
\begin{align*}
  B v_i &= \vphi_{i+1} v_{i+1}  && (0 \leq i \leq d), 
\end{align*}
where $v_{d+1}=0$ and $\vphi_{d+1}$ is an indeterminate.
Let $B^\flat$ denote the matrix that represents $B$ with respect to $\{v'_i\}_{i=0}^d$,
and $C^\natural$ denote the matrix that represents $C$ with respect to $\{v_i\}_{i=0}^d$.
By Lemma \ref{lem:C} $C^\natural$ is the anti-diagonal transpose of $B^\flat$.
Observe that $B^\flat$ is irreducible tridiagonal by Lemma \ref{lem:trid} and
since $B$ is irreducible tridiagonal on $\{V'_i\}_{i=0}^d$.
By these comments, $C^\natural$ is irreducible tridiagonal.
By this and Lemma \ref{lem:trid} $C$ is irreducible tridiagonal on $\{V_i\}_{i=0}^d$.
By the construction, $B$ raises $\{V_i\}_{i=0}^d$, so $B$ lowers $\{V_{d-i}\}_{i=0}^d$.
By \eqref{eq:Cvi2} $C$ raises $\{V'_i\}_{i=0}^d$, so $C$ lowers $\{V'_{d-i}\}_{i=0}^d$.
By these comments,
\begin{align*}
& \text{$B$ lowers $\{V_{d-i}\}_{i=0}^d$ and $C$ is irreducible tridiagonal on $\{V_{d-i}\}_{i=0}^d$}, \\
& \text{$C$ lowers $\{V'_{d-i}\}_{i=0}^d$ and $B$ is irreducible tridiagonal on $\{V'_{d-i}\}_{i=0}^d$}.
\end{align*}
Therefore $B,C$ is an LR by  Proposition \ref{prop:LRpair}.
By the assumption, $A,B$ is an LR pair.
By \eqref{eq:Avi2} and \eqref{eq:Cvi2}, $A,C$ is an LR pair with
$(A,C)$-decomposition $\{V'_{i}\}_{i=0}^d$.
Thus $A,B,C$ is an LR triple.
By \eqref{eq:Cvi2} the parameter sequence for $C,A$ is $\{\vphi_i\}_{i=1}^d$.
\end{proof}

\begin{proofof}{Theorem \ref{thm:main}}
(i)$\Rightarrow$(ii): Follows from Proposition \ref{prop:const}.

(ii)$\Rightarrow$(i): Follows from Lemma \ref{lem:based}.
\end{proofof}

\section{Proof of Theorem \ref{thm:main1}}
\label{sec:thmmain1}

In this section we prove Theorem \ref{thm:main1}.

\begin{lemma}    \label{lem:coincide}    \samepage
\ifDRAFT {\rm lem:coincide}. \fi
Let $A,B,C$ be an LR triple on $V$, and
let $\{V'_i\}_{i=0}^d$ be the $(A,C)$-decomposition.
Let $\{V''_i\}_{i=0}^d$ be a decomposition of $V$ that satisfy \eqref{eq:II}.
Then the following are equivalent:
\begin{itemize}
\item[\rm (i)]
$V'_d = V''_d$;
\item[\rm (ii)]
$\{V''_i\}_{i=0}^d$ is the $(B,C)$-decomposition.
\end{itemize}
\end{lemma}

\begin{proof}
(i)$\Rightarrow$(ii):
Let $\{\tilde{V}''_i\}_{i=0}^d$ be the $(B,C)$-decomposition.
We have $V''_i = \tilde{V}''_i$ for $0 \leq i \leq d$ by Lemma \ref{lem:Vicoincide}
and since each of $\{V''_i\}_{i=0}^d$, $\{\tilde{V}''_i\}_{i=0}^d$ is lowered by $B$.
The result follows.

(ii)$\Rightarrow$(i):
By Lemma \ref{lem:Braises}(ii), each of $V'_d$ and  $V''_d$ is the kernel of $C$.
\end{proof}

\begin{proofof}{Theorem \ref{thm:main1}}
(i)$\Rightarrow$(ii):
By Proposition \ref{prop:LRpair} $A,B$ is an LR pair.
By Theorem \ref{thm:main}(i)$\Rightarrow$(ii)
there exists an element $C$ in $\text{End}(V)$ such that
$A,B,C$ is an LR triple 
and $\{V'_i\}_{i=0}^d$ is the $(A,C)$-decomposition.
By this and Lemma \ref{lem:coincide}, $\{V''_i\}_{i=0}^d$
is the $(B,C)$-decomposition.
Thus (a)--(c) holds in Theorem \ref{thm:main1}(ii).

(ii)$\Rightarrow$(i):
By Lemma \ref{lem:based} the conditions \eqref{eq:I}, \eqref{eq:II} are satisfied.
By Lemma \ref{lem:coincide}  $V'_d=V''_d$.
\end{proofof}

\section{Bipartite LR triples}
\label{sec:bipartite}

In this section we recall the notion of a bipartite LR triple.
Let $A$ be an element in $\text{End}(V)$ and let $\{V_i\}_{i=0}^d$ be 
a decomposition of $V$.
We say $A$ is {\em zero-diagonal} on $\{V_i\}_{i=0}^d$ whenever
\begin{align*}
  A V_i &\subseteq V_0 + \cdots + V_{i-1} + V_{i+1} + \cdots + V_d   && (0 \leq i \leq d).
\end{align*}
Let $A,B,C$ be an LR triple on $V$.

\begin{definition}    \label{def:bipartite}    \samepage
\ifDRAFT {\rm def:biparite}. \fi
We say $A,B,C$ is {\em bipartite} whenever $A$ (resp.\ $B$) (resp.\ $C$)
is zero-diagonal on the $(B,C)$-decomposition
(resp.\ $(C,A)$-decomposition)  (resp.\ $(A,B)$-decomposition).
\end{definition}

\begin{lemma} {\rm  (see \cite[Example 13.3, Lemma 16.2]{T:LRT}.)}
\label{lem:d0}  \samepage
\ifDRAFT {\rm lem:d0}. \fi
Assume $d=0$.
Then each of $A,B,C$ is zero, and the LR triple $A,B,C$ is bipartite.
\end{lemma}

\begin{lemma}  {\rm (See \cite[Lemma 16.6]{T:LRT}.)}  
\label{lem:deven}   \samepage
\ifDRAFT {\rm lem:deven}. \fi
Assume $A,B,C$ is bipartite.
Then $d$ is even.
\end{lemma}

\begin{definition}       \label{def:idempotentdata}    \samepage
\ifDRAFT {\rm def:idempotentdata}. \fi
By the {\em idempotent data} for $A,B,C$ we mean the sequence
\begin{equation}              \label{eq:idempotent}
    (\{E_i\}_{i=0}^d, \{E'_i\}_{i=0}^d, \{E''_i\}_{i=0}^d),
\end{equation}
where $\{E_i\}_{i=0}^d$  (resp.\ $\{E'_i\}_{i=0}^d$) (resp.\ $\{E''_i\}_{i=0}^d$)
denotes the idempotent sequence for the
$(A,B)$-decomposition (resp.\ $(B,C)$-decomposition) (resp.\ $(C,A)$-decomposition).
\end{definition}

\begin{lemma}  {\rm (See \cite[Lemma 16.12]{T:LRT}.)}
\label{lem:Vout}    \samepage
\ifDRAFT {\rm lem:Vout}. \fi
Assume $A,B,C$ is bipartite.
Let \eqref{eq:idempotent} be the idempotent data for $A,B,C$.
Then the following $\F$-subspaces are equal:
\begin{align}                    \label{eq:Vout}
 & \sum_{j=0}^{d/2} E_{2j} V,   && \sum_{j=0}^{d/2} E'_{2j} V, && \sum_{j=0}^{d/2} E''_{2j} V,  
\end{align}
and the following $\F$-subspaces are equal:
\begin{align}                    \label{eq:Vin}
& \sum_{j=1}^{d/2} E_{2j-1} V,  &&  \sum_{j=1}^{d/2} E'_{2j-1} V,  &&  \sum_{j=1}^{d/2} E''_{2j-1} V.
\end{align}
\end{lemma}

\begin{definition}   \label{def:VoutVin}    \samepage
\ifDRAFT {\rm def:VoutVin}. \fi
With reference to Lemma \ref{lem:Vout},
let  $V_\text{out}$ and $V_\text{in}$ denote the common values
of \eqref{eq:Vout} and \eqref{eq:Vin}, respectively.
Observe that
\begin{align*}
  V &= V_\text{out} + V_\text{in}  && \text{(direct sum)}.
\end{align*}
\end{definition}

\begin{definition}    \label{def:XoutXin}    \samepage
\ifDRAFT {\rm def:XoutXin}. \fi
Let $X$ be an element in $\text{\rm End}(V)$.
With reference to Definition \ref{def:VoutVin},
we define elements $X_\text{out}$ and $X_\text{in}$ in $\text{End}(V)$ as follows.
The element $X_\text{out}$ acts on $V_\text{out}$ as $X$,
and on $V_\text{in}$ as zero.
The element $X_\text{in}$ acts on $V_\text{in}$ as $X$,
and on $V_\text{out}$ as zero.
Observe that
\[
   X = X_\text{out} + X_\text{in}.
\]
\end{definition}

\section{The parameter array and the Toeplitz data}
\label{sec:pre}

In this section we recall some data for an LR triple.
Let $A,B,C$ be an LR triple on $V$.

\begin{definition}  {\rm (See \cite[Definition 13.5]{T:LRT}.) }
\label{def:parray}    \samepage
\ifDRAFT {\rm def:parray}. \fi
By the {\em parameter array} for $A,B,C$ we mean the sequence
\begin{equation}                      \label{eq:parray}
    (\{\vphi_i\}_{i=1}^d, \{\vphi'_i\}_{i=1}^d, \{\vphi''_i\}_{i=1}^d),
\end{equation}
where  $\{\vphi_i\}_{i=1}^d$ (resp.\ $\{\vphi'_i\}_{i=1}^d$) (resp.\ $\{\vphi''_i\}_{i=1}^d$)
denotes the parameter sequence for the LR pair $A,B$  (resp.\ $B,C$) (resp.\ $C,A$).
\end{definition}

\begin{definition}   \label{def:compatible}    \samepage
\ifDRAFT {\rm def:compatible}. \fi
Let $\{v_i\}_{i=0}^d$ and $\{v'_i\}_{i=0}^d$ be bases for $V$.
We say $\{v_i\}_{i=0}^d$ and $\{v'_i\}_{i=0}^d$ are {\em compatible}
whenever $v_0 = v'_0$.
\end{definition}

\begin{definition}  {\rm (See \cite[Definition 13.44]{T:LRT}.)}
\label{def:T}    \samepage
\ifDRAFT {\rm def:T}. \fi
We define matrices $T,T',T''$ in $\Mat_{d+1}(\F)$ as follows:
\begin{itemize}
\item[\rm (i)]
$T$ is the transition matrix from an $(C,B)$-basis to the compatible
$(C,A)$-basis.
\item[\rm (ii)]
$T'$ is the transition matrix from an $(A,C)$-basis to the compatible
$(A,B)$-basis.
\item[\rm (iii)]
$T''$ is the transition matrix from an $(B,A)$-basis to the compatible
$(B,C)$-basis.
\end{itemize}
\end{definition}

\begin{lemma}  {\rm (See \cite[Lemma 13.41]{T:LRT}.) }
 \label{lem:T}   \samepage
\ifDRAFT {\rm lem:T}. \fi
With reference to Definition \ref{def:T}, the matrices $T,T',T''$
are upper triangular and Toeplitz.
\end{lemma}

\begin{definition}  {\rm (See \cite[Definition 13.45]{T:LRT}.)}
 \label{def:Toeplitzdata}    \samepage
\ifDRAFT {\rm def:Toeplitzdata}. \fi
Let $T,T',T''$ be from Definition \ref{def:T}.
By the {\em Toeplitz data} for $A,B,C$ we mean the sequence
\begin{equation}             \label{eq:Toeplitzdata}
  (\{\alpha_i\}_{i=0}^d, \{\alpha'_i\}_{i=0}^d, \{\alpha''_i\}_{i=0}^d),
\end{equation}
where $\{\alpha_i\}_{i=0}^d$ (resp.\ $\{\alpha'_i\}_{i=0}^d$)
(resp.\ $\{\alpha''_i\}_{i=0}^d$) are the parameters of $T$ (resp.\ $T'$) (resp.\ $T''$).
For notational convenience, define each of $\alpha_{d+1}$, $\alpha'_{d+1}$, $\alpha''_{d+1}$
is zero.
\end{definition}

For the rest of this section, 
let \eqref{eq:parray} and \eqref{eq:Toeplitzdata} be the parameter array
and the Toeplitz data for $A,B,C$, respectively.

\begin{lemma}  {\rm (See \cite[Lemma 13.46]{T:LRT}.)}
\label{lem:alpha0}    \samepage
\ifDRAFT {\rm lem:alpha0}. \fi
We have $\alpha_0 =1$, $\alpha'_0=1$, $\alpha''_0=1$.
\end{lemma}

\begin{lemma}   {\rm (See \cite[Lemmas 13.7, 13.22, 13.48]{T:LRT}.) }
 \label{lem:scalarmultiple}    \samepage
\ifDRAFT {\rm lem:scalarmultiple}. \fi
Let $\alpha, \beta, \gamma$ be nonzero scalars in $\F$.
Then $\alpha A, \beta B, \gamma C$ is an LR triple.
For this LR triple, 
\begin{itemize}
\item[\rm (i)]
the parameter array is
$ ( \{\alpha \beta \vphi_i\}_{i=1}^d,  \{ \beta \gamma \vphi'_i \}_{i=1}^d, 
\{ \gamma \alpha \vphi''_i\}_{i=1}^d)$;
\item[\rm (ii)]
the idempotent data is equal to the idempotent data for $A,B,C$;
\item[\rm (iii)]
the Toeplitz data is
$ (\{ \gamma^{-i} \alpha_i\}_{i=0}^d, \{\alpha^{-i} \alpha'_i \}_{i=0}^d, \{\beta^{-i} \alpha''_i\}_{i=0}^d)$.
\end{itemize}
\end{lemma}

\begin{lemma}  {\rm (See \cite[Lemma 16.5]{T:LRT}.)}
\label{lem:alpha1}   \samepage
\ifDRAFT {\rm lem:alpha1}. \fi
Assume that $A,B,C$ is nonbipartite.
Then $d \geq 1$.
Moreover each of
$\alpha_1$, $\alpha'_1$, $\alpha''_1$ is nonzero.
\end{lemma}

\begin{lemma} {\rm (See \cite[Corollary 14.5]{T:LRT}.)}  
\label{lem:alpha1vphii}    \samepage
\ifDRAFT {\rm lem:alpha1vphii}. \fi
For $1 \leq i \leq d$,
\begin{equation*}     
   \frac{\alpha_1}{\vphi_i} = \frac{\alpha'_1}{\vphi'_i}
  = \frac{\alpha''_1}{\vphi''_i}.
\end{equation*}
\end{lemma}

\begin{lemma}  {\rm (See \cite[Lemma 16.3]{T:LRT}.)}
\label{lem:alphaA}    \samepage
\ifDRAFT {\rm lem:alphaA}. \fi
Assume that $A,B,C$ is bipartite (resp.\ nonbipartite).
Let $\alpha,\beta,\gamma$ be nonzero scalars in $\F$.
Then the LR triple
$\alpha A, \beta B, \gamma C$
is bipartite (resp.\ nonbipartite).
\end{lemma}

\begin{lemma} {\rm (See \cite[Lemma 16.6]{T:LRT}.)}
\label{lem:bipartitealphai}    \samepage
\ifDRAFT {\rm lem:bipartitealphai}. \fi
Assume that $A,B,C$ is bipartite.
Then each of 
$\alpha_i$, $\alpha'_i$, $\alpha''_i$
is zero if $i$ is odd and nonzero if $i$ is even.
\end{lemma}

\begin{lemma}   {\rm (See \cite[Lemma 16.22]{T:LRT}.) }
\label{lem:alpha2vphii}    \samepage
\ifDRAFT {\rm lem:alpha2vphii}. \fi
Assume that $A,B,C$ is bipartite and $d \geq 2$.
Then the following hold for $1 \leq i,j \leq d$.
\begin{itemize}
\item[\rm (i)]
Assume that $i$, $j$ have opposite parity. Then
\begin{equation*} 
  \frac{\alpha_2} { \vphi_i \vphi_j } =
  \frac{\alpha'_2} { \vphi'_i \vphi'_j } =
  \frac{\alpha''_2} { \vphi''_i \vphi''_j }.
\end{equation*}
\item[\rm (ii)]
Assume that $i$, $j$ have the same parity. Then
\begin{equation*} 
  \frac{\vphi_i} { \vphi_j } =
  \frac{\vphi'_i} { \vphi'_j } =
 \frac{\vphi''_i} { \vphi''_j }. 
\end{equation*}
\end{itemize}
\end{lemma}

\begin{lemma}    {\rm (See \cite[Lemmas 16.31, 16.33]{T:LRT}.)}
\label{lem:bipartitescalar2}    \samepage
\ifDRAFT {\rm lem:bipartitescalar2}. \fi
Assume that $A,B,C$ is bipartite.
Let 
\[
 \alpha_\text{\rm out}, \quad
 \alpha_\text{\rm in}, \quad
 \beta_\text{\rm out}, \quad
 \beta_\text{\rm in}, \quad
 \gamma_\text{\rm out}, \quad
 \gamma_\text{\rm in}
\]
be nonzero scalars in $\F$.
Then the $3$-tuple
\[
 \alpha_\text{\rm out} A_\text{\rm out} + \alpha_\text{\rm in} A_\text{\rm in},  \qquad
 \beta_\text{\rm out} B_\text{\rm out} + \beta_\text{\rm in} B_\text{\rm in},  \qquad
 \gamma_\text{\rm out} C_\text{\rm out} + \gamma_\text{\rm in} C_\text{\rm in},
\]
is a bipartite LR triple.
For this LR triple,
\begin{itemize}
\item[\rm (i)]
the parameter array is
$(\{\vphi_i f_i\}_{i=1}^d, \{\vphi'_i f'_i\}_{i=1}^d, \{\vphi''_i f''_i\}_{i=1}^d)$,
where
\begin{align*}
  f_i &= \alpha_\text{\rm out} \beta_\text{\rm in},   &
  f'_i &= \beta_\text{\rm out} \gamma_\text{\rm in},   &
  f''_i &= \gamma_\text{\rm out} \alpha_\text{\rm in}   && \text{if $i$ is even},
\\
  f_i &= \alpha_\text{\rm in} \beta_\text{\rm out},   &
  f'_i &= \beta_\text{\rm in} \gamma_\text{\rm out},   &
  f''_i &= \gamma_\text{\rm in} \alpha_\text{\rm out}  && \text{if $i$ is odd};
\end{align*}
\item[\rm (ii)]
the idempotent data is equal to the idempotent data for $A,B,C$;
\item[\rm (iii)]
the Toeplitz data is
 $( \{\alpha_i g''_i\}_{i=0}^d, \{\alpha'_i g_i\}_{i=0}^d, \{\alpha''_i g'_i\}_{i=0}^d)$,
where
\begin{align*}
  g_i &= (\alpha_\text{\rm out} \alpha_\text{\rm in})^{-i/2},  &
  g'_i &= (\beta_\text{\rm out} \beta_\text{\rm in})^{-i/2},  &
  g''_i &= (\gamma_\text{\rm out} \gamma_\text{\rm in})^{-i/2} && \text{if $i$ is even},
\\
 g_i &=0, & g'_i &=0, & g''_i &= 0   && \text{if $i$ is odd}.
\end{align*}
\end{itemize}
\end{lemma}

\section{Proof of Theorems \ref{thm:nonbipartite}, \ref{thm:bipartite}, \ref{thm:main2}}
\label{sec:proofs}

In this section we prove Theorems \ref{thm:nonbipartite}, \ref{thm:bipartite}, \ref{thm:main2}.

\begin{lemma}    \label{lem:ACBCdecomp}    \samepage
\ifDRAFT {\rm lem:ACBCdecomp}. \fi
Let $A,B,C$ and $A,B,\tilde{C}$ be LR triples on $V$.
Then the following are equivalent:
\begin{itemize}
\item[\rm (i)]
the $(A,C)$-decomposition is equal to the $(A, \tilde{C})$-decomposition;
\item[\rm (ii)]
the  $(B,C)$-decomposition is equal to the $(B, \tilde{C})$-decomposition.
\end{itemize}
\end{lemma}

\begin{proof}
(i)$\Rightarrow$(ii):
Let $\{V'_i\}_{i=0}^d$ (resp.\ $\{\tilde{V}'_i\}_{i=0}^d$) be the $(A,C)$-decomposition 
(resp.\ $(A,\tilde{C})$-decomposition).
Let $\{V''_i\}_{i=0}^d$ (resp.\ $\{\tilde{V}''_i\}_{i=0}^d$) be the $(B,C)$-decomposition 
(resp.\ $(B,\tilde{C})$-decomposition).
Lemma \ref{lem:Braises}(ii), 
$V'_d = \text{\rm Ker} \, C$, 
$\tilde{V}'_d = \text{\rm Ker} \, \tilde{C}$,
$V''_d = \text{\rm Ker} \, C$, 
$\tilde{V}''_d = \text{\rm Ker} \, \tilde{C}$.
By the assumption, $V'_d = \tilde{V}'_d$.
By these comments, $V''_d = \tilde{V}''_d$.
By this and Lemma \ref{lem:Vicoincide} $V''_i = \tilde{V}''_i$ for $0 \leq i \leq d$.
Thus (ii) holds.

(ii)$\Rightarrow$(i):
Similar.
\end{proof}

We use the following notation.

\begin{notation}   \label{notation}    \samepage
\ifDRAFT {\rm notation}. \fi
Let $A,B,C$ and $A,B,\tilde{C}$ be LR triples on $V$.
For the LR triple $A,B,C$, we denote the parameter array and the Toeplitz data
by
\begin{align*}
 &  ( \{\vphi_i\}_{i=1}^d, \{\vphi'_i\}_{i=1}^d, \{\vphi''_i\}_{i=1}^d),
&&    ( \{\alpha_i\}_{i=0}^d, \{\alpha'_i\}_{i=0}^d, \{\alpha''_i\}_{i=0}^d),
\end{align*}
respectively.
For the LR triple $A,B, \tilde{C}$, we denote the parameter array and the Toeplitz data
by
\begin{align*}
 & ( \{\tilde{\vphi}_i\}_{i=1}^d, \{\tilde{\vphi}'_i \}_{i=1}^d, \{ \tilde{\vphi}''_i \}_{i=1}^d),
&&  ( \{ \tilde{\alpha}_i \}_{i=0}^d,  \{ \tilde{\alpha}'_i \}_{i=0}^d, \{ \tilde{\alpha}''_i \}_{i=0}^d),
\end{align*}
respectively.
\end{notation}

\begin{lemma}    \label{lem:vphi}     \samepage
\ifDRAFT {\rm lem:vphi}. \fi
With reference to Notation \ref{notation},
$\vphi_i = \tilde{\vphi}_i$ for $1 \leq i \leq d$.
\end{lemma}

\begin{proof}
Each of $\{\vphi_i\}_{i=1}^d$ and $\{\tilde{\vphi}_i\}_{i=1}^d$ is the
parameter sequence for the LR pair $A,B$.
\end{proof}

\begin{lemma}   \label{lem:alpha}    \samepage
\ifDRAFT {\rm lem:alpha}. \fi
With reference to Notation \ref{notation},
assume that the $(A,C)$-decomposition is equal
to the $(A, \tilde{C})$-decomposition.
Then $\alpha'_ i = \tilde{\alpha}'_i$ for $0 \leq i \leq d$.
\end{lemma}

\begin{proof}
Let $T'$ (resp.\ $\tilde{T}'$) 
be the upper triangular Toeplitz matrix with parameters $\{\alpha'_i\}_{i=0}^d$
(resp. $\{ \tilde{\alpha}'_i\}_{i=0}^d$).
Fix an $(A,B)$-basis $\{v_i\}_{i=0}^d$ for $V$.
Let $\{v'_i\}_{i=0}^d$ be the $(A,C)$-basis for $V$ such that $v'_0 = v_0$.
By the assumption and Definition \ref{def:ABbasis},
 $\{v'_i\}_{i=0}^d$ is an $(A, \tilde{C})$-basis for $V$.
By this and Definition \ref{def:T},
each of $T'$, $\tilde{T}'$ is the transition matrix from $\{v'_i\}_{i=0}^d$ to
$\{v_i\}_{i=0}^d$. So $T' = \tilde{T}'$.
The result follows.
\end{proof}

\begin{lemma}    \label{lem:bip}    \samepage
\ifDRAFT {\rm lem:bip}. \fi
With reference to Notation \ref{notation},
assume that the $(A,C)$-decomposition is equal
to the $(A, \tilde{C})$-decomposition.
Then $A,B,C$ is bipartite if and only if $A,B,\tilde{C}$ is bipartite.
\end{lemma}

\begin{proof}
By Lemma \ref{lem:d0}, we may assume $d \geq 1$.
By Lemmas \ref{lem:alpha1} and \ref{lem:bipartitealphai},  $A,B,C$ is bipartite 
if and only if $\alpha'_1=0$.
Similarly, $A,B,\tilde{C}$ is bipartite if and only if $\tilde{\alpha}'_1=0$.
By Lemma \ref{lem:alpha}, $\alpha'_1=0$ if and only if $\tilde{\alpha}'_1=0$.
The result follows form these comments.
\end{proof}

\begin{proofof}{Theorem \ref{thm:nonbipartite}}
(i)$\Rightarrow$(ii):
We use Notation \ref{notation}.
By Lemma \ref{lem:bip} the LR triple $A,B,\tilde{C}$ is nonbipartite.
Note by Lemma \ref{lem:alpha1} 
that $d \geq 1$ and each of $\alpha_1$, $\alpha'_1$, $\tilde{\alpha}_1$, $\tilde{\alpha}'_1$
is nonzero.
Define
\begin{equation}        \label{eq:gamma}
   \gamma = \alpha_1 \tilde{\alpha}_1^{-1}.
\end{equation}
Note that $\gamma \neq 0$.
Define
\begin{equation*} 
  \bar{C} = \gamma C.
\end{equation*}
By Lemma \ref{lem:scalarmultiple},
$A,B, \bar{C}$ is a nonbipartite LR triple.
For this LR triple, we denote the parameter array and the
Toeplitz data by
\begin{align*}
 &   (\{ \bar{\vphi}_i \}_{i=1}^d, 
    \{ \bar{\vphi}'_i \}_{i=1}^d, 
    \{ \bar{\vphi}''_i \}_{i=1}^d),
&& 
    (\{ \bar{\alpha}_i \}_{i=0}^d, 
    \{ \bar{\alpha}'_i \}_{i=0}^d, 
    \{ \bar{\alpha}''_i \}_{i=9}^d),
\end{align*}
respectively.
By Lemma \ref{lem:scalarmultiple},
the $(A,C)$-decomposition is equal to
the $(A, \bar{C})$-decomposition, and
\begin{align}
  \bar{\vphi}_i &=  \vphi_i , &&   (1 \leq i \leq d),   \label{eq:vphibar}
\end{align}
\begin{align}
  \bar{\alpha}_i &= \gamma^{-i} \alpha_i,  &
  \bar{\alpha}'_i &= \alpha'_i      &&   (0 \leq i \leq d).          \label{eq:alphabar}
\end{align}
Applying Lemma \ref{lem:alpha1vphii} to the LR triples $A,B, \tilde{C}$ and
$A,B, \bar{C}$,
\begin{align}
    \frac{ \tilde{\alpha}_1 } { \tilde{\vphi}_i }
  = \frac{\tilde{\alpha}'_1 } { \tilde{\vphi}'_i }   &&   (1 \leq i \leq d),  \label{eq:aux1}
\\
    \frac{ \bar{\alpha}_1 } { \bar{\vphi}_i }
  = \frac{\bar{\alpha}'_1 } { \bar{\vphi}'_i }   &&   (1 \leq i \leq d).  \label{eq:aux2}
\end{align}
By \eqref{eq:aux1} and Lemmas \ref{lem:vphi}, \ref{lem:alpha},
\begin{align}
   \tilde{\vphi}'_i &= \frac{ \alpha'_1 } { \tilde{\alpha} _1} \vphi_i  
                              && (1 \leq i \leq d).                              \label{eq:tildevphi}
\end{align}
By \eqref{eq:vphibar}, \eqref{eq:alphabar}, \eqref{eq:aux2},
\begin{align}
  \bar{\vphi}'_i &= \frac{ \alpha'_1} { \gamma^{-1} \alpha_1 } \vphi_i
                              &&  (1 \leq i \leq d).                            \label{eq:barvphi}
\end{align}
Combining \eqref{eq:gamma}, \eqref{eq:tildevphi}, \eqref{eq:barvphi},
\begin{align}
  \tilde{\vphi}'_i &= \bar{\vphi}_i  &&   (1 \leq i \leq d).
\end{align}
Therefore the LR pairs $B, \tilde{C}$ and $B, \bar{C}$  have the same parameter sequence.
By the construction, the $(A, \tilde{C})$-decomposition is equal to the
$(A, \bar{C})$-decomposition.
By this and Lemma \ref{lem:ACBCdecomp} the $(B, \tilde{C})$-decomposition
is equal to the $(B,\bar{C})$-decomposition.
By these comments and Lemma \ref{lem:AB},  $\tilde{C} = \bar{C}$.
Thus $\tilde{C} = \gamma C$.

(ii)$\Rightarrow$(i):
By Lemma \ref{lem:scalarmultiple}(ii).
\end{proofof}

\begin{proofof}{Theorem \ref{thm:bipartite}}
(i)$\Rightarrow$(ii):
We use Notation \ref{notation}.
By Lemma \ref{lem:bipartitealphai}, $d$ is even.
By Lemma \ref{lem:d0}, we may assume that $d \geq 2$.
By Lemma \ref{lem:bip}, the LR triple $A,B,\tilde{C}$ is bipartite.
Note by Lemma \ref{lem:bipartitealphai} that each of
$\alpha_2$, $\alpha'_2$, $\tilde{\alpha}_2$, $\tilde{\alpha}'_2$ is nonzero.
Define scalars
\begin{align}    \label{eq:gammaout}
 \gamma_\text{out} &= \frac{\tilde{\vphi}'_1} {\vphi'_1},  &
 \gamma_\text{in} &=
         \frac{\alpha_2 \vphi'_1} { \tilde{\alpha}_2 \tilde{\vphi}'_1 }.
\end{align}
Note that each of $\gamma_\text{out}$, $\gamma_\text{in}$ is nonzero.
Define  
\begin{equation*}  
  \bar{C} = \gamma_\text{out} C_\text{out} + \gamma_\text{in} C_\text{in}.
\end{equation*}
By Lemma \ref{lem:bipartitescalar2}, $A,B, \bar{C}$ is a bipartite LR triple.
For this LR triple, we denote the parameter array and the
Toeplitz data by
\begin{align*}
 &   (\{ \bar{\vphi}_i \}_{i=1}^d, 
    \{ \bar{\vphi}'_i \}_{i=1}^d, 
    \{ \bar{\vphi}''_i \}_{i=1}^d),
&& 
    (\{ \bar{\alpha}_i \}_{i=0}^d, 
    \{ \bar{\alpha}'_i \}_{i=0}^d, 
    \{ \bar{\alpha}''_i \}_{i=9}^d),
\end{align*}
respectively.
By Lemma \ref{lem:bipartitescalar2},
the $(A,C)$-decomposition is equal to
the $(A, \bar{C})$-decomposition, and
\begin{align}
  \bar{\vphi}_i &=  \vphi_i , 
& \bar{\vphi}'_i &=
    \begin{cases}
      \vphi'_i \gamma_\text{in}  &  \text{ if $i$ is even},  \\
      \vphi'_i \gamma_\text{out} & \text{ if $i$ is odd}
    \end{cases}
&&   (1 \leq i \leq d),   \label{eq:vphibar2}
\end{align}
\begin{align}
  \bar{\alpha}_{2i} &= (\gamma_\text{out} \gamma_\text{in})^{-i} \alpha_{2i},  &
  \bar{\alpha}'_{2i} &= \alpha'_{2i}      &&   (0 \leq i \leq d/2).          \label{eq:alphabar2}
\end{align}
Applying Lemma \ref{lem:alpha2vphii} to the LR triples $A,B,\tilde{C}$ and
$A,B, \bar{C}$,
\begin{align}
  \frac{ \tilde{\alpha}_2 } { \tilde{\vphi}_{2i} \tilde{\vphi}_1 } 
&= \frac{ \tilde{\alpha}'_2 } { \tilde{\vphi}'_{2i} \tilde{\vphi}'_1 } 
   &&  (1 \leq i \leq d/2),                                        \label{eq:tilde1}
\\
  \frac{\tilde{\vphi}_{2i-1} } { \tilde{\vphi}_1 }
&=  \frac{\tilde{\vphi}'_{2i-1} } { \tilde{\vphi}'_1 }
   &&  (1 \leq i \leq d/2),                                        \label{eq:tilde2}
\\
  \frac{ \bar{\alpha}_2 } { \bar{\vphi}_{2i} \bar{\vphi}_1 } 
&= \frac{ \bar{\alpha}'_2 } { \bar{\vphi}'_{2i} \bar{\vphi}'_1 } 
   &&  (1 \leq i \leq d/2),                                        \label{eq:bar1}
\\
  \frac{\bar{\vphi}_{2i-1} } { \bar{\vphi}_1 }
&=  \frac{\bar{\vphi}'_{2i-1} } { \bar{\vphi}'_1 }
   &&  (1 \leq i \leq d/2).                                        \label{eq:bar2}
\end{align}
By Lemmas \ref{lem:vphi} and \ref{lem:alpha},
\begin{align}
\tilde{\vphi}_i &= \vphi_i \qquad (1 \leq i \leq d),  
&  \tilde{\alpha}'_2 &= \alpha'_2.                     \label{eq:tilde3}
\end{align}
By \eqref{eq:vphibar2} and \eqref{eq:alphabar2},
\begin{align}
\bar{\vphi}_i &= \vphi_i  \qquad (1 \leq i \leq d),
 & \bar{\vphi}'_1 &= \vphi'_1 \gamma_\text{out}, 
 &  \bar{\alpha}_2 &= (\gamma_\text{out} \gamma_\text{in})^{-1} \alpha_2,
 & \bar{\alpha}'_2 &= \alpha'_2.
                                                                             \label{eq:bar3}
\end{align}
Simplify \eqref{eq:tilde1}--\eqref{eq:bar2} using \eqref{eq:tilde3}, \eqref{eq:bar3} to find
\begin{align}
 \frac{ \tilde{\alpha}_2 }{ \vphi_{2i} \vphi_1 }
  &= \frac{ \alpha'_2 } { \tilde{\vphi}'_{2i} \tilde{\vphi}'_1 }  
                     &&  (1 \leq i \leq d/2),                           \label{eq:tilde1b}
\\
 \frac{\vphi_{2i-1} }{ \vphi_1 } 
  &= \frac{\tilde{\vphi}'_{2i-1} }{\tilde{\vphi}'_1 }
                     &&   (1 \leq i \leq d/2),                          \label{eq:tilde2b}
\\
 \frac{(\gamma_\text{out} \gamma_\text{in})^{-1} \alpha_2 } {\vphi_{2i} \vphi_1 }
 &= \frac{ \alpha'_2 } { \bar{\vphi}'_{2i} \vphi'_1 \gamma_\text{out}} 
                      &&  (1 \leq i \leq d/2),                          \label{eq:bar1b}
\\
 \frac{ \vphi_{2i-1} }{ \vphi_1 }
  &= \frac{ \bar{\vphi}'_{2i-1} } { \vphi'_1 \gamma_\text{out} }
                      &&  (1 \leq i \leq d/2).                          \label{eq:bar2b}
\end{align}
By \eqref{eq:tilde1b} and \eqref{eq:bar1b},
\begin{align*}
  \frac{ \bar{\vphi}'_{2i} } { \tilde{\vphi}'_{2i} }
  &= \frac{ \tilde{\alpha}_2 \tilde{\vphi}'_1 } { \alpha_2 \vphi'_1 } \gamma_\text{in}
              &&    (1 \leq i \leq d/2).
\end{align*}
By this and \eqref{eq:gammaout},
\begin{align}
  \tilde{\vphi}'_{2i} &= \bar{\vphi}'_{2i}   &&    (1 \leq i \leq d/2).    \label{eq:even}
\end{align}
By \eqref{eq:tilde2b} and \eqref{eq:bar2b},
\begin{align*}
 \frac{ \bar{\vphi}'_{2i-1} }{ \tilde{\vphi}'_{2i-1} }
  &= \frac{\vphi'_1} { \tilde{\vphi}'_1 } \gamma_\text{out}
    &&  (1 \leq i \leq d/2).
\end{align*}
By this and \eqref{eq:gammaout},
\begin{align}
  \tilde{\vphi}'_{2i-1} &= \bar{\vphi}'_{2i-1}  &&  (1 \leq i \leq d/2).   \label{eq:odd}
\end{align}
By \eqref{eq:even} and \eqref{eq:odd},
$\tilde{\vphi}'_i = \bar{\vphi}'_i$ for $1 \leq i \leq d$.
Thus the LR pairs $B, \tilde{C}$ and $B, \bar{C}$ have the same parameter sequence.
By the construction, the $(A, \tilde{C})$-decomposition is equal to the
$(A, \bar{C})$-decomposition.
By this and Lemma \ref{lem:ACBCdecomp} the $(B, \tilde{C})$-decomposition
is equal to the $(B,\bar{C})$-decomposition.
By these comments and Lemma \ref{lem:AB}, $\tilde{C} = \bar{C}$.
Therefore
$\tilde{C} = \alpha_\text{out} C_\text{out} + \alpha_\text{in} C_\text{in}$.

(ii)$\Rightarrow$(i):
By Lemma \ref{lem:bipartitescalar2}(ii).
\end{proofof}

\begin{proofof}{Theorem \ref{thm:main2}}
(i):
First assume that there exists a nonzero scalar $\gamma$ in $\F$ such that
$\tilde{C} = \gamma C$.
By Lemma \ref{lem:scalarmultiple} $A,B, \gamma C$ is an LR triple, 
and the idempotent data for this LR triple is equal to the idempotent data for $A,B,C$.
So the $(A, \gamma C)$-decomposition is equal to the $(A,C)$-decomposition,
and the  $(B, \gamma C)$-decomposition is equal to the $(B,C)$-decomposition.
Thus  ($\tilde{\text{\rm a}}$)--($\tilde{\text{\rm c}}$) hold.
Next assume that $\tilde{C}$ satisfies ($\tilde{\text{\rm a}}$)--($\tilde{\text{\rm c}}$).
By (b) and ($\tilde{\text{\rm b}}$) the $(A, \tilde{C})$-decomposition is equal to the
$(A,C)$-decomposition.
Now by Theorem \ref{thm:nonbipartite} there exists a nonzero scalar $\gamma$ in $\F$
such that $\tilde{C} = \gamma C$.

(ii):
First assume  that there exist nonzero scalars $\gamma_\text{out}$, $\gamma_\text{in}$
in $\F$ such that
$\tilde{C} = \gamma_\text{out} C_\text{out} + \gamma_\text{in} C_\text{in}$.
By Lemma \ref{lem:bipartitescalar2} $A,B,\tilde{C}$ is an LR triple, 
and the idempotent data for this LR triple is equal to the idempotent data for $A,B,C$.
So the $(A,C)$-decomposition is equal to the $(A,\tilde{C})$-decomposition,
and  $(B, \gamma C)$-decomposition is equal to the $(B,C)$-decomposition.
Thus  ($\tilde{\text{\rm a}}$)--($\tilde{\text{\rm c}}$) hold.
Next assume that  $\tilde{C}$ satisfies ($\tilde{\text{\rm a}}$)--($\tilde{\text{\rm c}}$).
By (b) and ($\tilde{\text{\rm b}}$) the $(A, \tilde{C})$-decomposition is equal to the
$(A,C)$-decomposition.
Now by Theorem \ref{thm:bipartite} there exist nonzero scalar 
$\gamma_\text{out}$, $\gamma_\text{in}$ in $\F$ such that
$\tilde{C} = \gamma_\text{out} C_\text{out} + \gamma_\text{in} C_\text{in}$.
\end{proofof}

\section{Acknowledgement}

The author thanks Paul Terwilliger for many insightful comments that lead to
great improvements in the paper.

\bigskip\bigskip\noindent
Kazumasa Nomura\\
Tokyo Medical and Dental University\\
Kohnodai, Ichikawa, 272-0827 Japan\\
email: knomura@pop11.odn.ne.jp

\medskip\noindent
{\small
{\bf Keywords.} Lowering map, raising map, LR triple, quantum algebra, Lie algebra
\\
\noindent
{\bf 2010 Mathematics Subject Classification.} 17B37, 15A21
}

\end{document}